\documentclass{article}

\usepackage{amssymb,amsthm,amsmath}

\newcommand{\ZZ}{\mathbb{Z}}
\newcommand{\NN}{\mathbb{N}}

\newcommand{\FF}{\mathbb{F}}

\newtheorem{theorem}{Theorem}[section]
\newtheorem{lemma}[theorem]{Lemma}
\newtheorem{proposition}[theorem]{Proposition}
\newtheorem{corollary}[theorem]{Corollary}

\theoremstyle{definition}
\newtheorem{definition}[theorem]{Definition}

\newtheorem{remark}[theorem]{Remark}
\usepackage[square,numbers]{natbib}
\usepackage{hyperref}
\usepackage{breakurl}
\usepackage{url}
\hypersetup{pdfpagemode=none,colorlinks,citecolor=blue,filecolor=black,linkcolor=black,urlcolor=blue}
\newcommand{\vct}[1]{\mathbf{#1}}
\newcommand{\zint}{\,..\,}
\newcommand{\nth}{^{\text{th}}}
\DeclareMathOperator{\moddec}{mod}
\renewcommand{\mod}[1]{\,(\moddec #1)}
\usepackage{enumitem}
\usepackage{nameref}
\makeatletter
\newcommand{\clabel}[2]{\protected@write \@auxout {}{\string \newlabel {#1}{{#2}{\thepage}{#2}{#1}{}} }\hypertarget{#1}{}}
\makeatother
\newcommand{\lb}{\allowbreak}
\usepackage{etoolbox}
\newcommand{\breaklist}[2][,\lb]{\def\nextitem{\def\nextitem{#1}}\renewcommand*{\do}[1]{\nextitem{##1}}\docsvlist{#2}}
\DeclareMathOperator{\cirdec}{circ}		
	
\newcommand{\cir}[1]{\cirdec(\breaklist{#1})}
\newcommand{\pcir}[2]{\cirdec_{#1}(\breaklist{#2})}

\newcommand{\floor}[1]{\lfloor#1\rfloor}
\usepackage{empheq}
\newcommand{\BM}[1]{\boldmath {$#1$}\unboldmath}
\DeclareMathOperator{\autdec}{Aut}
\newcommand{\aut}[1]{\autdec(#1)}
\usepackage{booktabs}
\usepackage{upgreek}
\newcommand{\vctg}[1]{\boldsymbol{#1}}
\makeatletter
\renewcommand*\env@matrix[1][*\c@MaxMatrixCols c]{\hskip -\arraycolsep\let\@ifnextchar\new@ifnextchar\array{#1}}
\makeatother
\usepackage[font=footnotesize,labelfont=bf,textfont=normalfont,margin=0cm]{caption}
\usepackage{adjustbox}
\providecommand{\keywords}[1]{\small\textit{Keywords}: #1}
\providecommand{\msc}[1]{\small\textit{2020 MSC}: #1}

\title{New binary self-dual codes of lengths 56, 58, 64, 80 and 92 from a modification of the four circulant construction}

\author{J. Gildea, A. Korban and A. M. Roberts\\
Department of Mathematical and Physical Sciences\\
University of Chester\\
Thornton Science Park\\
Chester CH2 4NU\\
United Kingdom
}

\date{}

\begin{document}

\maketitle


\msc{94B05, 15B10, 15B33}

\let\thefootnote\relax\footnote{E-mail addresses: \href{mailto:j.gildea@chester.ac.uk}{j.gildea@chester.ac.uk} (J. Gildea), \href{mailto:adrian3@windowslive.com}{adrian3@windowslive.com} (A. Korban), \href{mailto:adammichaelroberts@outlook.com}{adammichaelroberts@outlook.com} (A. M. Roberts)}

\begin{abstract}
In this work, we give a new technique for constructing self-dual codes over commutative Frobenius rings using $\lambda$-circulant matrices. The new construction was derived as a modification of the well-known four circulant construction of self-dual codes. Applying this technique together with the building-up construction, we construct singly-even binary self-dual codes of lengths 56, 58, 64, 80 and 92 that were not known in the literature before. Singly-even self-dual codes of length 80 with $\beta\in\{2,4,5,6,8\}$ in their weight enumerators are constructed for the first time in the literature.
\end{abstract}

\maketitle

\section{Introduction}

Self-dual codes are types of linear codes which possess many interesting properties and are closely related to many other mathematical structures. Much research in particular has been invested into developing techniques for constructing new extremal binary self-dual codes. One of the most well-known and extensively applied of the techniques is the four circulant construction, which was first introduced in \cite{1} and uses a matrix $G$ defined by
	\begin{equation*}
	G=
	\begin{pmatrix}[c|c]
	I_{2n} & X
	\end{pmatrix},\quad\text{where }
	X=
	\begin{pmatrix}
	A & B\\-B^T & A^T
	\end{pmatrix},
	\end{equation*}
and where $A$ and $B$ are circulant matrices. It follows that $G$ is a generator matrix of a self-dual code of length $4n$ if and only if $AA^T+BB^T=-I_n$. In this work, we give a modification of the four circulant construction which we apply to construct extremal, optimal and best known binary self-dual codes that have previously not been known to exist. The new technique can be used to construct self-dual codes over any commutative Frobenius ring $R$. We introduce this technique and provide the conditions needed to produce a self-dual code.

For the proof of this technique, we utilise a specialised mapping $\Theta$ which was used in \cite{2}. This mapping is inherently associated with the matrix product $BA^T$, where $A$ and $B$ are $\lambda$-circulant matrices over $R$ such that $\lambda^2=1$. If $A$ is the $\lambda$-circulant matrix generated by $\vct{a}\in R^n$, then using $\Theta$ allows us to verify the equality $AA^T=-I_n$ by computing the values of $\floor{n/2}+1$ quantities in terms of $\vct{a}$. This eliminates the need to construct $A$ from its generating vector as well as computing the matrix product $AA^T$, which improves computational efficiency. We give and prove our own results concerning $\Theta$ as done so in \cite{2}.

Using the new technique together with the building-up construction, we find many self-dual codes with weight enumerator parameters of previously unknown values (relative to referenced sources). In total, 93 new codes are found, including
	\begin{enumerate}[label=$\bullet$]
	\item 29 singly-even binary self-dual $[56,28,10]$ codes;
	\item 1 binary self-dual $[58,29,10]$ code;
	\item 1 singly-even binary self-dual $[64,32,12]$ code;
	\item 50 singly-even binary self-dual $[80,40,14]$ codes;
	\item 12 binary self-dual $[92,46,16]$ codes.
	\end{enumerate}

The rest of the work is organised as follows. In Section 2, we give preliminary definitions and results on self-dual codes, Gray maps, circulant matrices, the specialised mapping $\Theta$ and the alphabets which we use. In Section 3, we introduce the new technique and conditions needed for producing a self-dual code. In Section 4, we apply the new technique and the building-up construction to obtain the new self-dual codes of lengths 56, 58, 64, 80 and 92, whose weight enumerator parameter values and automorphism group orders we detail. We also tabulate the results in this section. We finish with concluding remarks and discussion of possible expansion on this work.

\section{Preliminaries}

\subsection{Self-Dual Codes}

Let $R$ be a commutative Frobenius ring (see \cite{3} for a full description of Frobenius rings and codes over Frobenius rings). Throughout this work, we always assume $R$ has unity. A code $\mathcal{C}$ of length $n$ over $R$ is a subset of $R^n$ whose elements are called codewords. If $\mathcal{C}$ is a submodule of $R^n$, then we say that $\mathcal{C}$ is linear. Let $\mathbf{x},\mathbf{y}\in R^n$ where $\mathbf{x}=(x_1,x_2,\dots,x_n)$ and $\mathbf{y}=(y_1,y_2,\dots,y_n)$. The (Euclidean) dual $\mathcal{C}^{\bot}$ of $\mathcal{C}$ is given by
	\begin{equation*}
	\mathcal{C}^{\bot}=\{\mathbf{x}\in R^n: \langle\mathbf{x},\mathbf{y}\rangle=0,\forall\mathbf{y}\in\mathcal{C}\},
	\end{equation*}	
where $\langle,\rangle$ denotes the Euclidean inner product defined by
	\begin{equation*}
	\langle\mathbf{x},\mathbf{y}\rangle=\sum_{i=1}^nx_iy_i.
	\end{equation*}

We say that $\mathcal{C}$ is self-orthogonal if $\mathcal{C}\subseteq \mathcal{C}^\perp$ and self-dual if $\mathcal{C}=\mathcal{C}^{\bot}$.

An upper bound on the minimum (Hamming) distance of a doubly-even (Type II) binary self-dual code was given in \cite{5} and likewise for a singly-even (Type I) binary self-dual code in \cite{4}. Let $d_{\text{I}}(n)$ and $d_{\text{II}}(n)$ be the minimum distance of a Type I and Type II binary self-dual code of length $n$, respectively. Then
	\begin{equation*}
	d_{\text{II}}(n)\leq 4\floor{n/24}+4
	\end{equation*}
and
	\begin{equation*}
	d_{\text{I}}(n)\leq
	\begin{cases}
	4\floor{n/24}+2,& \text{if }n\equiv 0\pmod{24},\\
	4\floor{n/24}+4,& \text{if }n\not\equiv 22\pmod{24},\\
	4\floor{n/24}+6,& \text{if }n\equiv 22\pmod{24}.
	\end{cases}
	\end{equation*}

A self-dual code whose minimum distance meets its corresponding bound is called \textit{extremal}. A self-dual code with the highest possible minimum distance for its length is said to be \textit{optimal}. Extremal codes are necessarily optimal but optimal codes are not necessarily extremal. A \textit{best known} self-dual code is a self-dual code with the highest known minimum distance for its length.

\subsection{Alphabets}

In this paper, we consider the alphabets $\FF_2$, $\FF_2+u\FF_2$, $\FF_2+u\FF_2+v\FF_2+uv\FF_2$, $\FF_4$ and $\FF_4+u\FF_4$.

Define
	\begin{equation*}
	\FF_2+u\FF_2=\{a+bu:a,b\in\FF_2,u^2=0\}.
	\end{equation*}

Then $\FF_2+u\FF_2$ is a commutative ring of order 4 and characteristic 2 such that $\FF_2+u\FF_2\cong\FF_2[u]/\langle u^2\rangle$. 

Define
	\begin{equation*}
	\FF_2+u\FF_2+v\FF_2+uv\FF_2=\{a+bu+cv+duv:a,b,c,d\in\FF_2,u^2=v^2=uv+vu=0\}.
	\end{equation*}

Then $\FF_2+u\FF_2+v\FF_2+uv\FF_2$ is a commutative ring of order 16 and characteristic 2 such that $\FF_2+u\FF_2+v\FF_2+uv\FF_2\cong\FF_2[u,v]/\langle u^2,v^2,uv+vu\rangle$. Note that $\FF_2+u\FF_2+v\FF_2+uv\FF_2$ can be viewed as an extension of $\FF_2+u\FF_2$ and so we can also define
	\begin{equation*}
	\FF_2+u\FF_2+v\FF_2+uv\FF_2=\{a+bv: a,b\in \FF_2+u\FF_2,v^2=0\}.
	\end{equation*}

We define $\FF_4\cong\FF_2[\omega]/\langle \omega^2+\omega+1\rangle$ so that
	\begin{equation*}
	\FF_4=\{a{\omega}+b(1+\omega): a,b\in\FF_2,\omega^2+\omega+1=0\}.
	\end{equation*}

Define
	\begin{equation*}
	\FF_4+u\FF_4=\{a+bu: a,b\in\FF_4,u^2=0\}.
	\end{equation*}

Then $\FF_4+u\FF_4$ is a commutative ring of order 16 and characteristic 2 such that $\FF_4+u\FF_4\cong\FF_4[u]/\langle u^2\rangle\cong\FF_2[\omega,u]/\langle \omega^2+\omega+1,u^2,\omega u+u\omega\rangle$. Note that $\FF_4+u\FF_4$ can be viewed as an extension of $\FF_2+u\FF_2$ and so we can also define
	\begin{equation*}
	\FF_4+u\FF_4=\{a\omega+b(1+\omega): a,b\in \FF_2+u\FF_2,\omega^2+\omega+1=0\}.
	\end{equation*}

We recall the following Gray maps from \cite{6,7,8,35}:
	\begin{align*}
	\varphi_{\FF_2+u\FF_2}&:(\FF_2+u\FF_2)^n\to\FF_2^{2n}\\
        &\quad a+bu\mapsto(b,a+b),\,a,b\in\FF_2^{n},\\[6pt]
	\varphi_{\FF_4+u\FF_4}&:(\FF_4+u\FF_4)^n\to\FF_4^{2n}\\
		&\quad a+bu\mapsto(b,a+b),\,a,b\in\FF_4^{n},\\[6pt]
	\phi_{\FF_2+u\FF_2+v\FF_2+uv\FF_2}&:(\FF_2+u\FF_2+v\FF_2+uv\FF_2)^n\to (\FF_2+u\FF_2)^{2n}\\
		&\quad a+bv\mapsto(b,a+b),\,a,b\in (\FF_2+u\FF_2)^{n},\\[6pt]
	\psi_{\FF_4}&:\FF_4^n\to\FF_2^{2n}\\
		&\quad a\omega+b(1+\omega)\mapsto(a,b),\,a,b\in\FF_2^n,\\[6pt]
	\psi_{\FF_4+u\FF_4}&:(\FF_4+u\FF_4)^n\to (\FF_2+u\FF_2)^{2n}\\
		&\quad a\omega+b(1+\omega)\mapsto(a,b),\,a,b\in (\FF_2+u\FF_2)^n. 
	\end{align*}

Note that these Gray maps preserve orthogonality in their respective alphabets (see \cite{8} for details). If $\mathcal{C}\subseteq (\FF_4+u\FF_4)^n$, then the binary codes $\varphi_{\FF_2+u\FF_2}\circ\psi_{\FF_4+u\FF_4}(\mathcal{C})$ and $\psi_{\FF_4}\circ\varphi_{\FF_4+u\FF_4}(\mathcal{C})$ are equivalent to each
other (see \cite{7,8} for details). The Lee weight of a codeword is defined to be the Hamming weight of its binary image under any of the previously mentioned compositions of maps. A self-dual code in $R^n$ where $R$ is equipped with a Gray map to the binary Hamming space is said to be of Type II if the Lee weights of all codewords are multiples of 4, otherwise it is said to be of Type I.
	\begin{proposition}\label{prop-1}\textup{(\cite{8})}
	Let $\mathcal{C}$ be a code over $\FF_4+u\FF_4$. If $\mathcal{C}$ is self-orthogonal, then $\psi_{\FF_4+u\FF_4}(\mathcal{C})$ and $\varphi_{\FF_4+u\FF_4}(\mathcal{C})$ are self-orthogonal. The code $\mathcal{C}$ is a Type I (resp. Type II) code over $\FF_4+u\FF_4$ if and only if $\varphi_{\FF_4+u\FF_4}(\mathcal{C})$ is a Type I (resp. Type II) $\FF_4$-code if and only if $\psi_{\FF_4+u\FF_4}(\mathcal{C})$ is a Type I (resp. Type II) $(\FF_2+u\FF_2)$-code. Furthermore, the minimum Lee weight of $\mathcal{C}$ is the same as the minimum Lee weight of $\psi_{\FF_4+u\FF_4}(\mathcal{C})$ and $\varphi_{\FF_4+u\FF_4}(\mathcal{C})$.
	\end{proposition}

The next corollary follows immediately from Proposition \ref{prop-1}.
	\begin{corollary}
	Let $\mathcal{C}$ be a self-dual code over $\FF_4+u\FF_4$ of length $n$ and minimum Lee distance $d$. Then $\varphi_{\FF_2+u\FF_2}\circ\psi_{\FF_4+u\FF_4}(\mathcal{C})$ is a binary self-dual $[4n,2n,d]$ code. Moreover, the Lee weight enumerator of $\mathcal{C}$ is equal to the Hamming weight enumerator of $\varphi_{\FF_2+u\FF_2} \circ \psi_{\FF_4+u\FF_4}(\mathcal{C})$. If $\mathcal{C}$ is a Type I (resp. Type II) code, then $\varphi_{\FF_2+u\FF_2}\circ\psi_{\FF_4+u\FF_4}(\mathcal{C})$ is a Type I (resp. Type II) code.
	\end{corollary}

We also have the following proposition from \cite{36}:
	\begin{proposition}\textup{(\cite{36})}
	Let $\mathcal{C}$ be a self-dual code over $\FF_2+u\FF_2+v\FF_2+uv\FF_2$ of length $n$ and minimum Lee distance $d$. Then $\varphi_{\FF_2+u\FF_2}\circ\phi_{\FF_2+u\FF_2+v\FF_2+uv\FF_2}(\mathcal{C})$ is a binary $[4n,2n,d]$ self-dual code. Moreover, the Lee weight enumerator of $\mathcal{C}$ is equal to the Hamming weight enumerator of $\varphi_{\FF_2+u\FF_2}\circ\phi_{\FF_2+u\FF_2+v\FF_2+uv\FF_2}(\mathcal{C})$. If $\mathcal{C}$ is a Type I (resp. Type II) code, then $\varphi_{\FF_2+u\FF_2}\circ\phi_{\FF_2+u\FF_2+v\FF_2+uv\FF_2}(\mathcal{C})$ is a Type I (resp. Type II) code.
	\end{proposition}	

\subsection{Special Matrices}

We now recall the definitions and properties of some special matrices which we use in our work. Let $\vct{a}=(a_0,a_1,\ldots,a_{n-1})\in R^n$ where $R$ is a commutative ring and let
	\begin{equation*}
	A=\begin{pmatrix}
	a_0 & a_1 & a_2 & \cdots & a_{n-1}\\
	\lambda a_{n-1} & a_0 & a_1 & \cdots & a_{n-2}\\
	\lambda a_{n-2} & \lambda a_{n-1} & a_0 & \cdots & a_{n-3}\\
	\vdots & \vdots & \vdots & \ddots & \vdots\\
	\lambda a_1 & \lambda a_2 & \lambda a_3 & \cdots & a_0
	\end{pmatrix},
	\end{equation*}
where $\lambda\in R$. Then $A$ is called the $\lambda$-circulant matrix generated $\vct{a}$, denoted by $A=\pcir{\lambda}{\vct{a}}$. If $\lambda=1$, then $A$ is called the circulant matrix generated by $\vct{a}$ and is more simply denoted by $A=\cir{\vct{a}}$. If we define the matrix
	\begin{equation*}
	P_{\lambda}=\begin{pmatrix}
	\vct{0} & I_{n-1}\\
	\lambda & \vct{0}
	\end{pmatrix},
	\end{equation*}
then it follows that $A=\sum_{i=0}^{n-1}a_iP_{\lambda}^i$. Clearly, the sum of any two $\lambda$-circulant matrices is also a $\lambda$-circulant matrix. If $B=\pcir{\lambda}{\vct{b}}$ where $\vct{b}=(b_0,b_1,\ldots,b_{n-1})\in R^n$, then $AB=\sum_{i=0}^{n-1}\sum_{j=0}^{n-1}a_ib_jP_{\lambda}^{i+j}$. Since $P_{\lambda}^n=\lambda I_n$ there exist $c_k\in R$ such that $AB=\sum_{k=0}^{n-1}c_kP_{\lambda}^k$ so that $AB$ is also $\lambda$-circulant. In fact, it is true that
	\begin{equation*}
	c_{k}=\sum_{\substack{[i+j]_n=k\\i+j<n}}a_ib_j+\sum_{\substack{[i+j]_n=k\\i+j\geq n}}\lambda a_ib_j=\vct{x}_1\vct{y}_{k+1}
	\end{equation*}
for $k\in[0\zint n-1]$, where $\vct{x}_i$ and $\vct{y}_i$ respectively denote the $i\nth$ row and column of $A$ and $B$ and $[i+j]_n$ denotes the smallest non-negative integer such that $[i+j]_n\equiv i+j\mod{n}$. From this, we can see that $\lambda$-circulant matrices commute multiplicatively. Moreover, if $\lambda$ is a unit in $R$, then $A^T$ is $\lambda^{-1}$-circulant such that $A^T=a_0I_n+\lambda\sum_{i=1}^{n-1}a_{n-i}P_{\lambda^{-1}}^i$. It follows then that $AA^T$ is $\lambda$-circulant if and only if $\lambda$ is involutory in $R$, i.e. $\lambda^2=1$.

Let $J_n$ be an $n\times n$ matrix over $R$ whose $(i,j)\nth$ entry is $1$ if $i+j=n+1$ and 0 if otherwise. Then $J_n$ is called the $n\times n$ exchange matrix and corresponds to the row-reversed (or column-reversed) version of $I_n$. We see that $J_n$ is both symmetric and involutory, i.e. $J_n=J_n^T$ and $J_n^2=I_n$. For any matrix $A\in R^{m\times n}$, premultiplying $A$ by $J_m$ and postmultiplying $A$ by $J_n$ inverts the order in which the rows and columns of $A$ appear, respectively. Namely, the $(i,j)\nth$ entries of $J_mA$ and $AJ_n$ are the $([1-i]_m,j)\nth$ and $(i,[1-j]_n)\nth$ entries of $A$, respectively. Note that $[i+j]_n$ corresponds to the $(i+1,j+1)\nth$ entry of the matrix $J_nV$ where $V=\cir{n-1,0,1,\ldots,n-2}$ for $i,j\in[0\zint n-1]$.

\subsection{A Special Mapping}

We now introduce and explore the properties of a mapping which is inherently associated with the matrix product $BA^T$, where $A$ and $B$ are $\lambda$-circulant matrices such that $\lambda^2=1$. By utilising $\Theta$, we are able to improve the computational efficiency of our algorithms.
	\begin{definition}\textup{(\cite{2})}
	Let $R$ be a commutative ring and let $n\in\NN$ be fixed. Let $\Theta:R^n\times R^n\times\ZZ_n\to R$ be a mapping with an optional argument $\lambda\in R$ defined by
		\begin{equation*}
		\Theta(\vct{x},\vct{y},j)[\lambda]=\sum_{i=0}^{n-j-1}x_{[i+j]_n}y_i+\lambda\sum_{i=n-j}^{n-1}x_{[i+j]_n}y_i,
		\end{equation*}		
	where $\vct{x}=(x_0,x_1,\ldots,x_{n-1}),\vct{y}=(y_0,y_1,\ldots,y_{n-1})\in R^n$ and $j\in[0\zint n-1]$.
	
	If $j=0$, we define
		\begin{equation*}
		\Theta(\vct{x},\vct{y},0)=\sum_{i=0}^{n-1}x_iy_i=\vct{x}\vct{y}^T,
		\end{equation*}
	which is independent of $\lambda$.
	
	If $\lambda$ is unspecified, then we assume $\lambda=1$ so that
		\begin{equation*}
		\Theta(\vct{x},\vct{y},j)=\sum_{i=0}^{n-1}x_{[i+j]_n}y_i.
		\end{equation*}
	\end{definition}

	\begin{lemma}\label{lemm-1}\textup{(\cite{2})}
	Let $R$ be a commutative ring. Let $\vct{x},\vct{y}\in R^n$ and let $\lambda\in R:\lambda^2=1$. Then $\Theta(\vct{x},\vct{y},j)[\lambda]=\lambda\Theta(\vct{y},\vct{x},n-j)[\lambda]$, $\forall j\in[0\zint n-1]$.
	\end{lemma}
	
	\begin{proof}
	If $\vct{x}=(x_0,x_1,\ldots,x_{n-1})$, then $x_{[i+k]_n}=\tilde{x}_s$, where $\tilde{\vct{x}}$ is the vector $\vct{x}$ after being circularly shifted by $k$ places for some $k\in[0\zint n-1]$. If $\vct{y}=(y_0,y_1,\ldots,y_{n-1})$, then in a correspondence between the elements $x_i$ and $y_i$, inflicting a circular shift to both $\vct{x}$ and $\vct{y}$ by the same number of places preserves this correspondence. Thus, noting that $\lambda^2=1$ by assumption, we have
		\begin{align*}
		\lambda\Theta(\vct{y},\vct{x},n-j)[\lambda]&=\lambda\left(\sum_{i=0}^{n-(n-j)-1}y_{[i+(n-j)]_n}x_i+\lambda\sum_{i=n-(n-j)}^{n-1}y_{[i+(n-j)]_n}x_i\right)\\
		&=\lambda\sum_{i=0}^{j-1}y_{[i+(n-j)]_n}x_i+\sum_{i=j}^{n-1}y_{[i+(n-j)]_n}x_i\\
		&=\sum_{i=j}^{n-1}x_{[i+j+(n-j)]_n}y_{[i+(n-j)]_n}+\lambda\sum_{i=0}^{j-1}x_{[i+j+(n-j)]_n}y_{[i+(n-j)]_n}\\
		&=\sum_{i=0}^{n-j-1}x_{[i+j]_n}y_i+\lambda\sum_{i=n-j}^{n-1}x_{[i+j]_n}y_i\\
		&=\Theta(\vct{x},\vct{y},j)[\lambda].
		\end{align*}
	\end{proof}
	
	\begin{remark}\label{remark-1}
	In Lemma \ref{lemm-1}, suppose we want to calculate $f(j)=\Theta(\vct{x},\vct{x},j)[\lambda]$, $\forall j\in[0\zint n-1]$. We have $f(j)=\lambda\Theta(\vct{x},\vct{x},n-j)[\lambda]$ which, since $\lambda^2=1$, implies
		\begin{equation*}
		f(n-j)=\Theta(\vct{x},\vct{x},n-j)[\lambda]=\lambda\Theta(\vct{x},\vct{x},j)[\lambda]=\lambda f(j)
		\end{equation*}
	so that $f(j)=\lambda f(n-j)$. Therefore, to calculate $f(j)$ for $j\in[0\zint n-1]$, it is sufficient to determine $f(j)$ for $j\in[0\zint\floor{n/2}]$.
	\end{remark}
	
	\begin{lemma}\label{lemm-2}\textup{(\cite{2})}
	Let $R$ be a commutative ring. Let $A=\pcir{\lambda}{\vct{a}}$ and $B=\pcir{\lambda}{\vct{b}}$ with $\vct{a},\vct{b}\in R^n$ and $\lambda\in R:\lambda^2=1$. Then $BA^T=\pcir{\lambda}{v_0,v_1,\ldots,v_{n-1}}$, where $v_j=\Theta(\vct{b},\vct{a},j)[\lambda]$, $\forall j\in[0\zint n-1]$.
	\end{lemma}
	
	\begin{proof}
	Since $\lambda^2=1$ by assumption, we know that $BA^T$ is $\lambda$-circulant such that $BA^T=\pcir{\lambda}{\vct{x}_1\vct{y}_1^T,\vct{x}_1\vct{y}_2^T,\ldots,\vct{x}_1\vct{y}_n^T}$, where $\vct{x}_i$ and $\vct{y}_i$ denote the $i\nth$ rows of $B$ and $A$, respectively. Let $BA^T=\pcir{\lambda}{v_0,v_1,\ldots,v_{n-1}}$ so that $v_j=\vct{x}_1\vct{y}_{j+1}^T$ for $j\in[0\zint n-1]$. It is easy to see that $v_0=\sum_{i=0}^{n-1}b_ia_i$. In the product $v_1=\vct{x}_1\vct{y}_2^T$, we see that the indices of the vector $\vct{y}_2=(\lambda a_{n-1},a_0,a_1,\ldots,a_{n-2})$ correspond to the indices of the vector $\vct{x}_1=(b_0,b_1,b_2,\ldots,b_{n-1})$ after being circularly shifted to the right by 1 place. Thus, in $v_1=\vct{x}_1\vct{y}_2^T$, there is a summation of terms in the form $b_{[i+1]_n}a_i$ for $i\in[0\zint n-1]$. By extending this argument, we see that in the product $\vct{x}_1\vct{y}_{j+1}^T$, there is a summation of terms in the form $b_{[i+j]_n}a_i$ for $i\in[0\zint n-1]$ and $j\in[1\zint n-1]$. Also, in $v_1$, we see that the terms of the summation will acquire $\lambda$ as a coefficient for $i=n-1$. By extending this argument, in $v_j$, we see that the terms of the summation will acquire $\lambda$ as a coefficient for $i\in[n-j\zint n-1]$ and $j\in[1\zint n-1]$. In summary, we have
		\begin{equation*}
		v_j=\begin{cases}
		\sum_{i=0}^{n-1}b_ia_i,& j=0,\\
		\sum_{i=0}^{n-j-1}b_{[i+j]_n}a_i+\lambda\sum_{i=n-j}^{n-1}b_{[i+j]_n}a_i,& j\in[1\zint n-1]
		\end{cases}
		\end{equation*}	
	and so, in terms of $\Theta$, we see that $v_j=\Theta(\vct{b},\vct{a},j)[\lambda]$, $\forall j\in[0\zint n-1]$.
	\end{proof}

With these lemmas established, we will now look at how $\Theta$ can be used to prove that a matrix $G=(I_n\,|\,A)$ is a generator matrix of a self-dual code.
	\begin{proposition}\label{prop-3}\textup{(\cite{2})}
	Let $R$ be a commutative ring. Let $A=\pcir{\lambda}{\vct{a}}$ with $\vct{a}\in R^n$ and $\lambda\in R:\lambda^2=1$. Then $AA^T=-I_n$ if and only if
		\begin{equation*}
		\Theta(\vct{a},\vct{a},j)[\lambda]=\begin{cases}
		-1,& j=0,\\
		0,& j\in[1\zint\floor{n/2}].
		\end{cases}
		\end{equation*}	
	\end{proposition}
	
	\begin{proof}
	Since $\lambda^2=1$ by assumption, by Lemma \ref{lemm-2} we have $AA^T=\pcir{\lambda}{v_0,v_1,\ldots,v_{n-1}}$ where $v_j=\Theta(\vct{a},\vct{a},j)[\lambda]$, $\forall j\in[0\zint n-1]$. The main diagonal and off-diagonal entries of $AA^T$ are given by $v_0$ and $v_j$, respectively, for $j\neq 0$, so $AA^T=-I_n$ if and only if $v_0=-1$ and $v_j=0$, $\forall j\in[1\zint n-1]$. By Remark \ref{remark-1}, we see that it is sufficient to verify $v_0=-1$ and $v_j=0$, $\forall j\in[1\zint\floor{n/2}]$. Therefore, we see that $AA^T=-I_n$ if and only if
		\begin{equation*}
		\Theta(\vct{a},\vct{a},j)[\lambda]=\begin{cases}
		-1,& j=0,\\
		0,& j\in[1\zint\floor{n/2}].
		\end{cases}
		\end{equation*}	
	\end{proof}

For example, consider the pure double circulant construction of self-dual codes given by $G=(I_n\,|\,A)$ for an $n\times n$ $\lambda$-circulant matrix $A=\pcir{\lambda}{\vct{a}}$ over a commutative Frobenius ring $R$ such that $\lambda\in R:\lambda^2=1$. We know that $G$ is a generator matrix of a self-dual $[2n,n]$ code over $R$ if and only if $AA^T=-I_n$ and in terms of $\Theta$, by Proposition \ref{prop-3} this is true if and only if
	\begin{equation*}
	\Theta(\vct{a},\vct{a},j)[\lambda]=\begin{cases}
	-1,& j=0,\\
	0,& j\in[1\zint\floor{n/2}].
	\end{cases}
	\end{equation*}		
	
If $A$ is the $\lambda$-circulant matrix generated by $\vct{a}\in R^n$, then Proposition \ref{prop-3} tells us that in order to verify the equality $AA^T=-I_n$, we need only compute the values of $\floor{n/2}+1$ quantities in terms of $\vct{a}$ instead of having to construct $A$ from its generating vector as well as computing the product $AA^T$. We implement this result in our algorithms in order to reduce the number of computations required and hence improve efficiency.
\section{The Construction}

In this section, we present the new technique for constructing self-dual codes. We will hereafter always assume $R$ is a commutative Frobenius ring.
	\begin{theorem}\label{thm-1}
	Let
		\begin{equation*}
		G=\begin{pmatrix}[c|c]
		I_{2n} & X
		\end{pmatrix},\quad\text{where }
		X=\begin{pmatrix}
		-A^TCJ & -B\\
		B^TCJ & -A
		\end{pmatrix}
		\end{equation*}	
	and where $J=J_n$, $A=\pcir{\lambda}{\vct{a}}$, $B=\pcir{\lambda}{\vct{b}}$ and $C=\pcir{\mu}{\vct{c}}$ with $\vct{a},\vct{b},\vct{c}\in R^n$ and $\lambda,\mu\in R:\lambda^2=\mu^2=1$. Then $G$ is a generator matrix of a self-dual $[4n,2n]$ code over $R$ if and only if
		\begin{empheq}[left=\empheqlbrace]{align*}
		\sum_{\vct{x}\in S}\Theta(\vct{x},\vct{x},j)[\lambda]&=\begin{cases}
		-1,& j=0,\\
		0,& j\in[1\zint\floor{n/2}],
		\end{cases}\\
		\Theta(\vct{c},\vct{c},j)[\mu]&=\begin{cases}
		1,& j=0,\\
		0,& j\in[1\zint\floor{n/2}],
		\end{cases}
		\end{empheq}		
	where $S=\{\vct{a},\vct{b}\}$.	
	\end{theorem}
	
	\begin{proof}
	We know that $G$ is a generator matrix of a self-dual $[4n,2n]$ code over $R$ if and only if $XX^T=-I_{2n}$. Since $\lambda^2=1$ by assumption, we have that $A$ and $B$ as well as their transpositions all commute with one another multiplicatively. Firstly, since $J$ is symmetric we have
		\begin{equation*}
		X^T=\begin{pmatrix}
		-A^TCJ & -B\\
		B^TCJ & -A
		\end{pmatrix}^T=
		\begin{pmatrix}
		-JC^TA & JC^TB\\
		-B^T & -A^T
		\end{pmatrix}.
		\end{equation*}	
		
	If the $(i,j)\nth$ block-wise entry of $XX^T$ is $x_{i,j}$, noting that $J$ is involutory we see that
		\begin{align*}
		x_{1,1}&=A^TCJ^2C^TA+BB^T=A^TCC^TA+BB^T,\\
		x_{1,2}&=-A^TCJ^2C^TB+BA^T=-A^TCC^TB+BA^T,\\
		x_{2,1}&=-B^TCJ^2C^TA+AB^T=-B^TCC^TA+AB^T,\\
		x_{2,2}&=B^TCJ^2C^TB+AA^T=B^TCC^TB+AA^T.
		\end{align*}	
		
	Noting that $XX^T=-I_{2n}$ if and only if $A^TCC^TB=BA^T$, we see that
		\begin{align*}
		XX^T=-I_{2n}&\iff A^TCC^TA+BB^T=-I_n\\&\iff A^TCC^TAA^T+(BA^T)B^T=-A^T\\&\iff A^TCC^TAA^T+(A^TCC^TB)B^T=-A^T\\&\iff A^TCC^T(AA^T+BB^T)=-A^T\\&\iff
		CC^T(AA^T+BB^T)=-I_n
		\end{align*}
	and so, combined with our other required conditions, we have that $XX^T=-I_{2n}$ if and only if
		\begin{empheq}[left=\empheqlbrace]{align*}
		A^TCC^TA+BB^T&=-I_n,\\
		B^TCC^TB+AA^T&=-I_n,\\
		CC^T(AA^T+BB^T)&=-I_n.
		\end{empheq}	
		
	Clearly, we must have $CC^T=I_n$ for all of these equations to be satisfied. With this prerequisite, our conditions reduce to
		\begin{align*}
		x_{1,1}&=AA^T+BB^T,\\
		x_{1,2}&=\vct{0},\\
		x_{2,1}&=\vct{0},\\
		x_{2,2}&=AA^T+BB^T.
		\end{align*}
		
	Therefore, $XX^T=-I_{2n}$ if and only if $AA^T+BB^T=-I_n$ and $CC^T=I_n$ and by Proposition \ref{prop-3} this is true if and only if
		\begin{empheq}[left=\empheqlbrace]{align*}
		\sum_{\vct{x}\in S}\Theta(\vct{x},\vct{x},j)[\lambda]&=\begin{cases}
		-1,& j=0,\\
		0,& j\in[1\zint\floor{n/2}],
		\end{cases}\\
		\Theta(\vct{c},\vct{c},j)[\mu]&=\begin{cases}
		1,& j=0,\\
		0,& j\in[1\zint\floor{n/2}],
		\end{cases}
		\end{empheq}		
	where $S=\{\vct{a},\vct{b}\}$.	
	\end{proof}

	\begin{remark}
		Let $U'$ denote the set of involutory units in $R$ and let $N_{C}=N_C(R,n)$ denote the number of orthogonal $\mu$-circulant matrices $C$ (i.e. $CC^T=I_n$) over $R$ for all $\mu\in U'$. The search field for self-dual $[4n,2n]$ codes over $R$ constructed by Theorem \ref{thm-1} is of size $|R|^{2n}\cdot |U'|\cdot N_C$. In general, $N_C$ is relatively small, for example $N_C(\FF_2,20)=2,560$ and $N_C(\FF_2+u\FF_2+v\FF_2+uv\FF_2,5)=20,480$.
	\end{remark}
	\begin{remark}\label{remark-2}
	In Theorem \ref{thm-1}, we are in fact able to assume $C$ is any matrix over $R$ such that $C$ is orthogonal. Moreover, $C$ and $C^T$ need not commute multiplicatively with either $A$ or $B$ or their transpositions.
	\end{remark}
	
\section{Results}

In this section, we apply Theorem \ref{thm-1} to obtain the following types of new codes
	\begin{enumerate}[label=$\bullet$]
	\item best known singly-even binary self-dual codes of lengths 56 and 80;
	\item an extremal singly-even binary self-dual code of length 64;
	\item extremal binary self-dual codes of length 92.
	\end{enumerate}	

We also apply the following well-known technique for constructing self-dual codes referred to as the building-up construction.
	\begin{theorem}\label{thm-2}\textup{(\cite{33})}
	Let $R$ be a commutative Frobenius ring. Let $G'$ be a generator matrix of a self-dual $[2n,n]$ code $\mathcal{C}'$ over $R$ and let $\vct{r}_i$ denote the $i\nth$ row of $G'$. Let $\varepsilon\in R:\varepsilon^2=-1$, $\vctg{\updelta}\in R^{2n}:\langle\vctg{\updelta},\vctg{\updelta}\rangle=-1$ and $\gamma_i=\langle\vct{r}_i,\vctg{\updelta}\rangle$ for $i\in[1\zint n]$. Then the matrix
		\begin{equation*}
		G=\begin{pmatrix}[cc|c]
		1 & 0 & \vctg{\updelta}\\\hline
		-\gamma_1 & \varepsilon\gamma_1 & \vct{r}_1\\
		-\gamma_2 & \varepsilon\gamma_2 & \vct{r}_2\\
		\vdots & \vdots & \vdots\\
		-\gamma_n & \varepsilon\gamma_n & \vct{r}_n
		\end{pmatrix}
		\end{equation*}
	is a generator matrix of a self-dual $[2(n+1),n+1]$ code over $R$.
	\end{theorem}
	
By utilising Theorem \ref{thm-2}, we obtain the following types of new codes
	\begin{enumerate}[label=$\bullet$]
	\item best known singly-even binary self-dual codes of length 56;
	\item an optimal binary self-dual code of length 58.
	\end{enumerate}	

We conduct the search for these codes using MATLAB and determine their properties using Q-extension \cite{34}. In MATLAB, we employ an algorithm which randomly searches for the construction parameters that satisfy the necessary and sufficient conditions stated in Theorem \ref{thm-1}. For such parameters, we then build the corresponding binary generator matrices and print them to text files. We then use Q-extension to read these text files and determine the minimum distance, partial weight enumerator and automorphism group order of each corresponding code. We implement a similar procedure for Theorem \ref{thm-2}. We do not give any additional information regarding the automorphism groups. A database of generator matrices of the new codes is given online at \cite{gmd}. The database is partitioned into text files (interpretable by Q-extension) corresponding to each code type. In these files, specific properties of the codes including the construction parameters, weight enumerator parameter values and automorphism group order are formatted as comments above the generator matrices. Partial weight enumerators of the codes are also formatted as comments below the generator matrices. Table \ref{table-01} gives the hexadecimal notation system we use to represent elements of $\FF_2+u\FF_2$, $\FF_2+u\FF_2+v\FF_2+uv\FF_2$, $\FF_4$ and $\FF_4+u\FF_4$.

	\begin{table}[h!]
	\caption{Hexadecimal notation system for elements of $\FF_2+u\FF_2$, $\FF_2+u\FF_2+v\FF_2+uv\FF_2$, $\FF_4$ and $\FF_4+u\FF_4$.}\label{table-01}
	\centering
	\begin{adjustbox}{max width=\textwidth}
	\footnotesize
	\begin{tabular}{ccccc}\midrule
	$\FF_2+u\FF_2$ & $\FF_2+u\FF_2+v\FF_2+uv\FF_2$ & $\FF_4$ & $\FF_4+u\FF_4$ & Symbol\\\midrule
	$0$ & $0$ & $0$ & $0$ & \texttt{0}\\
	$1$ & $1$ & $1$ & $1$ & \texttt{1}\\
	$u$ & $u$ & $w$ & $w$ & \texttt{2}\\
	$1+u$ & $1+u$ & $1+w$ & $1+w$ & \texttt{3}\\
	$-$ & $v$ & $-$ & $u$ & \texttt{4}\\
	$-$ & $1+v$ & $-$ & $1+u$ & \texttt{5}\\
	$-$ & $u+v$ & $-$ & $w+u$ & \texttt{6}\\
	$-$ & $1+u+v$ & $-$ & $1+w+u$ & \texttt{7}\\
	$-$ & $uv$ & $-$ & $wu$ & \texttt{8}\\
	$-$ & $1+uv$ & $-$ & $1+wu$ & \texttt{9}\\
	$-$ & $u+uv$ & $-$ & $w+wu$ & \texttt{A}\\
	$-$ & $1+u+uv$ & $-$ & $1+w+wu$ & \texttt{B}\\
	$-$ & $v+uv$ & $-$ & $u+wu$ & \texttt{C}\\
	$-$ & $1+v+uv$ & $-$ & $1+u+wu$ & \texttt{D}\\
	$-$ & $u+v+uv$ & $-$ & $w+u+wu$ & \texttt{E}\\
	$-$ & $1+u+v+uv$ & $-$ & $1+w+u+wu$ & \texttt{F}\\\midrule
	\end{tabular}
	\end{adjustbox}
	\end{table}
	
\subsection{New Singly-Even Binary Self-Dual \BM{[56,28,10]} Codes}

The possible weight enumerators of a singly-even binary self-dual $[56,28,10]$ code are given in \cite{10} as
	\begin{align*}
	W_{56,1}&=1+(308+4\alpha)x^{10}+(4246-8\alpha)x^{12}+\cdots,\\
	W_{56,2}&=1+(308+4\alpha)x^{10}+(3990-8\alpha)x^{12}+\cdots,
	\end{align*}
where $\alpha\in\ZZ$. The existence of codes with weight enumerator $W_{56,1}$ has previously been determined for
	\begin{enumerate}[label=]
	\item $\alpha\in\{-z:z=12,\lb14,\lb16,\lb18,\lb19,\lb21,\lb...,\lb34,\lb38,\lb40,\lb51,\lb64\}$
	\end{enumerate}	
(see \cite{10,11,12,2}) and the existence of codes with weight enumerator $W_{56,2}$ has previously been determined for
	\begin{enumerate}[label=]
	\item $\alpha\in\{-z:z=0,\lb2,\lb4,\lb...,\lb32,\lb34,\lb35,\lb38,\lb40,\lb42,\lb56\}$
	\end{enumerate}
(see \cite{10,11,12,2}).

We obtain 29 new best known singly-even binary self-dual codes of length 56 of which 15 have weight enumerator $W_{56,1}$ for
	\begin{enumerate}[label=]
	\item $\alpha\in\{-z:z=13,\lb15,\lb17,\lb20,\lb35,\lb36,\lb37,\lb39,\lb41,\lb...,\lb44,\lb46,\lb48,\lb52\}$
	\end{enumerate}	
and 14 have weight enumerator $W_{56,2}$ for
	\begin{enumerate}[label=]
	\item $\alpha\in\{-z:z=1,\lb3,\lb33,\lb36,\lb37,\lb39,\lb41,\lb43,\lb...,\lb49\}$.
	\end{enumerate}	

Of the 29 new codes, one is constructed by applying Theorem \ref{thm-1} over $\FF_4$ (Table \ref{table-56-1}); 20 are constructed by first applying Theorem \ref{thm-1} to obtain codes of length 12 over $\FF_2+u\FF_2+v\FF_2+uv\FF_2$ (Table \ref{table-56-2}) to which we then apply Theorem \ref{thm-2} (Table \ref{table-56-4}) and similarly 8 are constructed by first applying Theorem \ref{thm-1} to obtain codes of length 12 over $\FF_4+u\FF_4$ (Table \ref{table-56-3}) to which we then apply Theorem \ref{thm-2} (Table \ref{table-56-4}).
	
	\begin{table}[h!]
	\caption{New singly-even binary self-dual $[56,28,10]$ code from Theorem \ref{thm-1} over $\FF_4$.}\label{table-56-1}
	\centering
	\begin{adjustbox}{max width=\textwidth}
	\footnotesize
	\begin{tabular}{ccccccccc}\midrule
	$\mathcal{C}_{56,i}$ & $\vct{a}$ & $\vct{b}$ & $\vct{c}$ & $W_{56,j}$ & $\alpha$ & $|\aut{\mathcal{C}_{56,i}}|$\\\midrule
	1 & \texttt{(1110320)} & \texttt{(3002312)} & \texttt{(3231112)} & 2 & $-49$ & $2\cdot 7$\\\midrule
	\end{tabular}
	\end{adjustbox}
	\end{table}
	
	\begin{table}[h!]
	\caption{Codes of length 12 over $\FF_2+u\FF_2+v\FF_2+uv\FF_2$ from Theorem \ref{thm-1} to which we apply Theorem \ref{thm-2} to obtain new singly-even binary self-dual $[56,28,10]$ codes.}\label{table-56-2}
	\centering
	\begin{adjustbox}{max width=\textwidth}
	\footnotesize
	\begin{tabular}{cccccc}\midrule
	$\mathcal{C}_{12,i}'$ & $\lambda$ & $\mu$ & $\vct{a}$ & $\vct{b}$ & $\vct{c}$\\\midrule
	1 & \texttt{F} & \texttt{9} & \texttt{(957)} & \texttt{(D85)} & \texttt{(EFE)}\\
	2 & \texttt{3} & \texttt{D} & \texttt{(831)} & \texttt{(71B)} & \texttt{(E1E)}\\
	3 & \texttt{1} & \texttt{7} & \texttt{(07F)} & \texttt{(371)} & \texttt{(66D)}\\
	4 & \texttt{B} & \texttt{B} & \texttt{(6B1)} & \texttt{(5DF)} & \texttt{(454)}\\
	5 & \texttt{D} & \texttt{B} & \texttt{(AF3)} & \texttt{(D3B)} & \texttt{(C43)}\\
	6 & \texttt{3} & \texttt{F} & \texttt{(33D)} & \texttt{(DF0)} & \texttt{(A2B)}\\
	7 & \texttt{1} & \texttt{3} & \texttt{(07F)} & \texttt{(F5B)} & \texttt{(C7C)}\\
	8 & \texttt{B} & \texttt{3} & \texttt{(FD5)} & \texttt{(075)} & \texttt{(E7E)}\\
	9 & \texttt{F} & \texttt{9} & \texttt{(99B)} & \texttt{(71C)} & \texttt{(C3C)}\\
	10 & \texttt{B} & \texttt{9} & \texttt{(359)} & \texttt{(25F)} & \texttt{(E3E)}\\
	11 & \texttt{1} & \texttt{B} & \texttt{(B71)} & \texttt{(F78)} & \texttt{(E6B)}\\
	12 & \texttt{5} & \texttt{1} & \texttt{(BBB)} & \texttt{(9A9)} & \texttt{(2D2)}\\\midrule
	\end{tabular}
	\end{adjustbox}
	\end{table}
	
	\begin{table}[h!]
	\caption{Codes of length 12 over $\FF_4+u\FF_4$ from Theorem \ref{thm-1} to which we apply Theorem \ref{thm-2} to obtain new singly-even binary self-dual $[56,28,10]$ codes.}\label{table-56-3}
	\centering
	\begin{adjustbox}{max width=\textwidth}
	\footnotesize
	\begin{tabular}{cccccc}\midrule
	$\mathcal{C}_{12,i}'$ & $\lambda$ & $\mu$ & $\vct{a}$ & $\vct{b}$ & $\vct{c}$\\\midrule
	13 & \texttt{9} & \texttt{D} & \texttt{(2A9)} & \texttt{(4AE)} & \texttt{(544)}\\
	14 & \texttt{D} & \texttt{D} & \texttt{(526)} & \texttt{(282)} & \texttt{(100)}\\
	15 & \texttt{9} & \texttt{5} & \texttt{(1A0)} & \texttt{(F2B)} & \texttt{(100)}\\
	16 & \texttt{9} & \texttt{5} & \texttt{(977)} & \texttt{(34F)} & \texttt{(100)}\\
	17 & \texttt{1} & \texttt{9} & \texttt{(736)} & \texttt{(CE1)} & \texttt{(858)}\\
	18 & \texttt{9} & \texttt{D} & \texttt{(1B5)} & \texttt{(529)} & \texttt{(001)}\\\midrule
	\end{tabular}
	\end{adjustbox}
	\end{table}
	
	\begin{table}[h!]
	\caption{New singly-even binary self-dual $[56,28,10]$ codes from applying Theorem \ref{thm-2} to $\mathcal{C}_{12,j}'$ as given in Tables \ref{table-56-2} and \ref{table-56-3}.}\label{table-56-4}
	\centering
	\begin{adjustbox}{max width=\textwidth}
	\footnotesize
	\begin{tabular}{ccccccc}\midrule
	$\mathcal{C}_{56,i}$ & $\mathcal{C}'_{12,j}$ & $\varepsilon$ & $\vctg{\updelta}$ & $W_{56,k}$ & $\alpha$ & $|\aut{\mathcal{C}_{56,i}}|$\\\midrule
	2 & 1 & \texttt{D} & \texttt{(EBEB4DA6D9A6)} & 1 & $-52$ & $2^2$\\
	3 & 1 & \texttt{D} & \texttt{(5D1373E2FDE5)} & 1 & $-48$ & $2^2$\\
	4 & 2 & \texttt{7} & \texttt{(C8A98C86755D)} & 1 & $-46$ & $2^2$\\
	5 & 3 & \texttt{D} & \texttt{(53DFED023D37)} & 1 & $-44$ & $2^2$\\
	6 & 13 & \texttt{D} & \texttt{(6847C689DE95)} & 1 & $-43$ & $2$\\
	7 & 1 & \texttt{7} & \texttt{(77EF377F3D46)} & 1 & $-42$ & $2^2$\\
	8 & 13 & \texttt{9} & \texttt{(A6A8CB9C1499)} & 1 & $-41$ & $2$\\
	9 & 14 & \texttt{9} & \texttt{(1155EF4F171C)} & 1 & $-39$ & $2$\\
	10 & 15 & \texttt{9} & \texttt{(071C9E507475)} & 1 & $-37$ & $2$\\
	11 & 4 & \texttt{B} & \texttt{(222332B06AF3)} & 1 & $-36$ & $2^2$\\
	12 & 16 & \texttt{D} & \texttt{(40DD2E682BA8)} & 1 & $-35$ & $2$\\
	13 & 5 & \texttt{1} & \texttt{(8EFCE02FED71)} & 1 & $-20$ & $2^2$\\
	14 & 17 & \texttt{9} & \texttt{(2396462B0954)} & 1 & $-17$ & $2$\\
	15 & 16 & \texttt{D} & \texttt{(200E34345CD9)} & 1 & $-15$ & $2$\\
	16 & 18 & \texttt{D} & \texttt{(0625805ED91E)} & 1 & $-13$ & $2$\\
	17 & 6 & \texttt{9} & \texttt{(B7B77FF08521)} & 2 & $-48$ & $2^2$\\
	18 & 7 & \texttt{5} & \texttt{(EEC4CDD6BDE3)} & 2 & $-47$ & $2^2$\\
	19 & 8 & \texttt{D} & \texttt{(C204459424ED)} & 2 & $-46$ & $2^2$\\
	20 & 9 & \texttt{1} & \texttt{(7141B8B30F5F)} & 2 & $-45$ & $2^2$\\
	21 & 10 & \texttt{D} & \texttt{(11AFA2912E93)} & 2 & $-44$ & $2^2$\\
	22 & 6 & \texttt{1} & \texttt{(6538469B9516)} & 2 & $-43$ & $2^2$\\
	23 & 1 & \texttt{F} & \texttt{(E10A9596C3BF)} & 2 & $-41$ & $2^2$\\
	24 & 7 & \texttt{F} & \texttt{(63A97D46ACCB)} & 2 & $-39$ & $2^2$\\
	25 & 10 & \texttt{1} & \texttt{(0137F87FD723)} & 2 & $-37$ & $2^2$\\
	26 & 8 & \texttt{B} & \texttt{(CDA9F7DEB302)} & 2 & $-36$ & $2^2$\\
	27 & 11 & \texttt{3} & \texttt{(43D72C348FC4)} & 2 & $-33$ & $2^2$\\
	28 & 3 & \texttt{1} & \texttt{(01A517F0F84B)} & 2 & $-3$ & $2^2$\\
	29 & 12 & \texttt{1} & \texttt{(A169C21CFB80)} & 2 & $-1$ & $2^2$\\\midrule
	\end{tabular}
	\end{adjustbox}
	\end{table}
	
\subsection{New Binary Self-Dual \BM{[58,29,10]} Code}

The possible weight enumerators of a binary self-dual $[58,29,10]$ code are given in \cite{42,37} as
	\begin{align*}
	W_{58,1}&=1+55x^{10}+5188x^{12}+\cdots,\\
	W_{58,2}&=1+(319-2\alpha-24\beta)x^{10}+(3132+2\alpha+152\beta)x^{12}+\cdots,
	\end{align*}
where $\alpha\in[0\zint 159-12\beta]$ and $\beta\in[0\zint 11]$ for $W_{58,2}$. The existence of codes with weight enumerator $W_{58,2}$ has previously been determined for
	\begin{enumerate}[label=]
	\item $\beta=0$ and $\alpha\in\{2z:z=0,\lb...,\lb66,\lb68,\lb71,\lb79\}$;
	\item $\beta=1$ and $\alpha\in\{2z:z=8,\lb...,\lb58,\lb63\}$;
	\item $\beta=2$ and $\alpha\in\{2z:z=0,\lb4,\lb6,\lb...,\lb55\}$
	\end{enumerate}	
(see \cite{38,39}).

We obtain one new optimal binary self-dual code of length 58 which has weight enumerator $W_{58,2}$ for
	\begin{enumerate}[label=]
	\item $\beta=1$ and $\alpha=118$.
	\end{enumerate}	

To construct the new code, we first apply Theorem \ref{thm-1} to obtain a code of length 28 over $\FF_4$ (Table \ref{table-58-1}) to the image of which under $\psi_{\FF_4}$ we then apply Theorem \ref{thm-2} (Table \ref{table-58-2}).

	\begin{table}[h!]
	\caption{Code of length 28 over $\FF_4$ from Theorem \ref{thm-1} to which we apply Theorem \ref{thm-2} to obtain a new binary self-dual $[58,29,10]$ code.}\label{table-58-1}
	\centering
	\begin{adjustbox}{max width=\textwidth}
	\footnotesize
	\begin{tabular}{cccc}\midrule
	$\mathcal{C}_{28,i}'$ & $\vct{a}$ & $\vct{b}$ & $\vct{c}$\\\midrule
	1 & \texttt{(2331311)} & \texttt{(2221222)} & \texttt{(0000001)}\\\midrule
	\end{tabular}
	\end{adjustbox}
	\end{table}

	\begin{table}[h!]
	\caption{New binary self-dual $[58,29,10]$ code from applying Theorem \ref{thm-2} to $\psi_{\FF_4}(\mathcal{C}_{28,j}')$ as given in Table \ref{table-58-1}.}\label{table-58-2}
	\centering
	\begin{adjustbox}{max width=\textwidth}
	\footnotesize
	\begin{tabular}{ccccccc}\midrule
	$\mathcal{C}_{58,i}$ & $\mathcal{C}_{28,j}'$ & $\vctg{\updelta}$ & $W_{58,k}$ & $\alpha$ & $\beta$ & $|\aut{\mathcal{C}_{58,i}}|$\\\midrule
	1 & 1 & \texttt{(00100011110001010001100101111101001001111010110001010100)} & 2 & $118$ & 1 & $2$\\\midrule
	\end{tabular}
	\end{adjustbox}
	\end{table}

\subsection{New Singly-Even Binary Self-Dual \BM{[64,32,12]} Code}

The possible weight enumerators of a singly-even binary self-dual $[64,32,12]$ code are given in \cite{42,37} as
	\begin{align*}
	W_{64,1}&=1+(1312+16\alpha)x^{12}+(22016-64\alpha)x^{14}+\cdots,\\
	W_{64,2}&=1+(1312+16\alpha)x^{12}+(23040-64\alpha)x^{14}+\cdots,
	\end{align*}
where $\alpha\in[14\zint 284]$ for $W_{64,1}$ and $\alpha\in[0\zint 277]$ for $W_{64,2}$. The existence of codes with weight enumerator $W_{64,2}$ has previously been determined for
	\begin{enumerate}[label=]
	\item $\alpha\in\{0,\lb...,\lb42,\lb44,\lb...,\lb52,\lb54,\lb...,\lb58,\lb60,\lb62,\lb64,\lb65,\lb69,\lb72,\lb80,\lb88,\lb96,\lb104,\lb108,\lb112,\lb114,\lb118,\lb120,\lb184\}$
	\end{enumerate}	
(see \cite{40,41,39,25,26}).

We obtain one new extremal singly-even binary self-dual code of length 64 which has weight enumerator $W_{64,2}$ for
	\begin{enumerate}[label=]
	\item $\alpha=53$.
	\end{enumerate}	

The new code is constructed by applying Theorem \ref{thm-1} over $\FF_4$ (Table \ref{table-64-1}).

	\begin{table}[h!]
	\caption{New singly-even binary self-dual $[64,32,12]$ code from Theorem \ref{thm-1} over $\FF_4$.}\label{table-64-1}
	\centering
	\begin{adjustbox}{max width=\textwidth}
	\footnotesize
	\begin{tabular}{ccccccc}\midrule
	$\mathcal{C}_{64,i}$ & $\vct{a}$ & $\vct{b}$ & $\vct{c}$ & $W_{64,j}$ & $\alpha$ & $|\aut{\mathcal{C}_{64,i}}|$\\\midrule
	1 & \texttt{(23113202)} & \texttt{(10112022)} & \texttt{(33100231)} & 2 & $53$ & $2^3$\\\midrule
	\end{tabular}
	\end{adjustbox}
	\end{table}

\subsection{New Singly-Even Binary Self-Dual \BM{[80,40,14]} Codes}

The weight enumerator of a singly-even binary self-dual $[80,40,14]$ code is given in \cite{23} as
	\begin{align*}
	W_{80}=1+(3200+4\alpha)x^{14}+(47645-8\alpha+256\beta)x^{16}+\cdots,
	\end{align*}
where $\alpha,\beta\in\ZZ$. The existence of codes with weight enumerator $W_{80}$ has previously been determined for
	\begin{enumerate}[label=]
	\item $\beta=0$ and $\alpha\in\{-z:z=34,\lb51,\lb68,\lb85,\lb102,\lb119,\lb136,\lb153,\lb160,\lb170,\lb180,\lb187,\lb200,\lb204,\lb220,\lb221,\lb238,\lb240,\lb255,\lb260,\lb272,\lb289,\lb300,\lb306,\lb323,\lb340,\lb357,\lb374,\lb391,\lb408,\lb425,\lb459\}$;
	\item $\beta=1$ and $\alpha\in\{-6z:z=16,\lb25,\lb28,\lb31,\lb34,\lb37,\lb40,\lb43,\lb52\}$;
	\item $\beta=10$ and $\alpha\in\{-2z:z=102,\lb138,\lb140,\lb147,\lb165,\lb174,\lb180,\lb183,\lb200,\lb210,\lb220\}$;
	\item $\beta=18$ and $\alpha\in\{-z:z=211,\lb229,\lb249,\lb256,\lb274,\lb287,\lb306,\lb310,\lb325,\lb355,\lb363,\lb401\}$
	\end{enumerate}	
(see \cite{24,23,11,25,26,2}).

We obtain 50 new best known singly-even binary self-dual codes of length 80 which have weight enumerator $W_{80}$ for
	\begin{enumerate}[label=]
	\item $\beta=0$ and $\alpha\in\{-z:z=400,\lb420,\lb440,\lb460\}$;
	\item $\beta=1$ and $\alpha\in\{-6z:z=26\}$;
	\item $\beta=\textbf{2}$ and $\alpha\in\{-4z:z=53,\lb54,\lb55,\lb59,\lb60,\lb64,\lb66,\lb67,\lb72,\lb73,\lb74\}$;
	\item $\beta=\textbf{4}$ and $\alpha\in\{-4z:z=41,\lb54,\lb65,\lb68,\lb88\}$;
	\item $\beta=\textbf{5}$ and $\alpha\in\{-10z:z=19,\lb...,\lb22,\lb24,\lb25,\lb26,\lb29\}$;
	\item $\beta=\textbf{6}$ and $\alpha\in\{-4z:z=66,\lb69,\lb71,\lb77,\lb79,\lb81,\lb82,\lb89,\lb91\}$;
	\item $\beta=\textbf{8}$ and $\alpha\in\{-4z:z=69,\lb73,\lb88,\lb92\}$;
	\item $\beta=10$ and $\alpha\in\{-2z:z=190,\lb230,\lb240,\lb260\}$;
	\item $\beta=18$ and $\alpha\in\{-z:z=384,\lb376,\lb364,\lb360\}$.
	\end{enumerate}

Until this point, no singly-even binary self-dual $[80,40,14]$ codes with weight enumerator $W_{80}$ for $\beta\in\{2,4,5,6,8\}$ had been known to exist.

Of the 50 new codes, 4 are constructed by applying Theorem \ref{thm-1} over $\FF_2$ (Table \ref{table-80-1}); 3 are constructed by applying Theorem \ref{thm-1} over $\FF_2+u\FF_2$ (Table \ref{table-80-2}); 34 are constructed by applying Theorem \ref{thm-1} over $\FF_2+u\FF_2+v\FF_2+uv\FF_2$ (Table \ref{table-80-3}) and 9 are constructed by applying Theorem \ref{thm-1} over $\FF_4$ (Table \ref{table-80-4}).
	
	\begin{table}[h!]
	\caption{New singly-even binary self-dual $[80,40,14]$ codes from Theorem \ref{thm-1} over $\FF_2$.}\label{table-80-1}
	\centering
	\begin{adjustbox}{max width=\textwidth}
	\footnotesize
	\begin{tabular}{ccccccc}\midrule
	$\mathcal{C}_{80,i}$ & $\vct{a}$ & $\vct{b}$ & $\vct{c}$ & $\alpha$ & $\beta$ & $|\aut{\mathcal{C}_{80,i}}|$\\\midrule
	1 & \texttt{(00110110101100111001)} & \texttt{(01111111101101111110)} & \texttt{(01111111000111111101)} & $-440$ & 0 & $2^{3}\cdot 5$\\
	2 & \texttt{(01010110110001001110)} & \texttt{(00101000110101010011)} & \texttt{(00111000100010100010)} & $-400$ & 0 & $2^{3}\cdot 5$\\
	3 & \texttt{(01100011100111001011)} & \texttt{(11100111100000100110)} & \texttt{(10100000001010001000)} & $-520$ & 10 & $2^{4}\cdot 3\cdot 5$\\
	4 & \texttt{(11011010010111001011)} & \texttt{(11110111101100100110)} & \texttt{(01011111001010000001)} & $-380$ & 10 & $2^{2}\cdot 5$\\\midrule
	\end{tabular}
	\end{adjustbox}
	\end{table}
	
	\begin{table}[h!]
	\caption{New singly-even binary self-dual $[80,40,14]$ codes from Theorem \ref{thm-1} over $\FF_2+u\FF_2$.}\label{table-80-2}
	\centering
	\begin{adjustbox}{max width=\textwidth}
	\footnotesize
	\begin{tabular}{ccccccccc}\midrule
	$\mathcal{C}_{80,i}$ & $\lambda$ & $\mu$ & $\vct{a}$ & $\vct{b}$ & $\vct{c}$ & $\alpha$ & $\beta$ & $|\aut{\mathcal{C}_{80,i}}|$\\\midrule
	5 & \texttt{1} & \texttt{1} & \texttt{(1012003233)} & \texttt{(1313102320)} & \texttt{(1212130203)} & $-420$ & 0 & $2^{2}\cdot 5$\\
	6 & \texttt{1} & \texttt{1} & \texttt{(2232033031)} & \texttt{(3023203111)} & \texttt{(0303023332)} & $-480$ & 10 & $2^{4}\cdot 5$\\
	7 & \texttt{1} & \texttt{1} & \texttt{(0130132012)} & \texttt{(0321112301)} & \texttt{(2020202320)} & $-460$ & 10 & $2^{2}\cdot 5$\\\midrule
	\end{tabular}
	\end{adjustbox}
	\end{table}
	
	\begin{table}[h!]
	\caption{New singly-even binary self-dual $[80,40,14]$ codes from Theorem \ref{thm-1} over $\FF_2+u\FF_2+v\FF_2+uv\FF_2$.}\label{table-80-3}
	\centering
	\begin{adjustbox}{max width=\textwidth}
	\footnotesize
	\begin{tabular}{ccccccccc}\midrule
	$\mathcal{C}_{80,i}$ & $\lambda$ & $\mu$ & $\vct{a}$ & $\vct{b}$ & $\vct{c}$ & $\alpha$ & $\beta$ & $|\aut{\mathcal{C}_{80,i}}|$\\\midrule
	8 & \texttt{B} & \texttt{3} & \texttt{(92EB2)} & \texttt{(00337)} & \texttt{(00030)} & $-156$ & 1 & $2^{3}\cdot 3$\\
	9 & \texttt{5} & \texttt{D} & \texttt{(51F8F)} & \texttt{(F5BF9)} & \texttt{(E9ECC)} & $-296$ & \textbf{2} & $2^{2}$\\
	10 & \texttt{B} & \texttt{B} & \texttt{(51B78)} & \texttt{(D7191)} & \texttt{(EC461)} & $-292$ & \textbf{2} & $2^{2}$\\
	11 & \texttt{9} & \texttt{1} & \texttt{(3FFFF)} & \texttt{(1813B)} & \texttt{(0440F)} & $-288$ & \textbf{2} & $2^{2}$\\
	12 & \texttt{7} & \texttt{7} & \texttt{(5793F)} & \texttt{(585D3)} & \texttt{(34A2C)} & $-268$ & \textbf{2} & $2^{2}$\\
	13 & \texttt{B} & \texttt{9} & \texttt{(11B93)} & \texttt{(D1BF8)} & \texttt{(3A00A)} & $-264$ & \textbf{2} & $2^{2}$\\
	14 & \texttt{F} & \texttt{7} & \texttt{(FDB18)} & \texttt{(991BD)} & \texttt{(64B46)} & $-256$ & \textbf{2} & $2^{2}$\\
	15 & \texttt{1} & \texttt{9} & \texttt{(B3D1D)} & \texttt{(3D8D3)} & \texttt{(EA9AE)} & $-240$ & \textbf{2} & $2^{2}$\\
	16 & \texttt{F} & \texttt{7} & \texttt{(55B39)} & \texttt{(35D8D)} & \texttt{(E2D2E)} & $-236$ & \textbf{2} & $2^{2}$\\
	17 & \texttt{D} & \texttt{D} & \texttt{(8B95F)} & \texttt{(5FDD9)} & \texttt{(2D244)} & $-220$ & \textbf{2} & $2^{2}$\\
	18 & \texttt{7} & \texttt{F} & \texttt{(F95FF)} & \texttt{(5F93C)} & \texttt{(14C4C)} & $-216$ & \textbf{2} & $2^{3}$\\
	19 & \texttt{7} & \texttt{3} & \texttt{(B8B3D)} & \texttt{(FDB9B)} & \texttt{(C3C4C)} & $-212$ & \textbf{2} & $2^{2}$\\
	20 & \texttt{5} & \texttt{7} & \texttt{(53710)} & \texttt{(93999)} & \texttt{(2A292)} & $-352$ & \textbf{4} & $2^{2}$\\
	21 & \texttt{D} & \texttt{D} & \texttt{(D8D1D)} & \texttt{(1BB77)} & \texttt{(3C6EC)} & $-272$ & \textbf{4} & $2^{2}$\\
	22 & \texttt{D} & \texttt{5} & \texttt{(5757D)} & \texttt{(97358)} & \texttt{(E070E)} & $-260$ & \textbf{4} & $2^{2}$\\
	23 & \texttt{1} & \texttt{F} & \texttt{(28FDB)} & \texttt{(E0FE5)} & \texttt{(2AA9A)} & $-216$ & \textbf{4} & $2^{2}$\\
	24 & \texttt{F} & \texttt{B} & \texttt{(5A4B6)} & \texttt{(8D538)} & \texttt{(B66EE)} & $-164$ & \textbf{4} & $2^{2}$\\
	25 & \texttt{3} & \texttt{B} & \texttt{(35853)} & \texttt{(7BFD9)} & \texttt{(0E603)} & $-364$ & \textbf{6} & $2^{2}$\\
	26 & \texttt{3} & \texttt{7} & \texttt{(975B5)} & \texttt{(8D137)} & \texttt{(0C405)} & $-356$ & \textbf{6} & $2^{2}$\\
	27 & \texttt{F} & \texttt{7} & \texttt{(73D97)} & \texttt{(38397)} & \texttt{(0A205)} & $-328$ & \textbf{6} & $2^{2}$\\
	28 & \texttt{5} & \texttt{D} & \texttt{(1DF31)} & \texttt{(9D565)} & \texttt{(600E5)} & $-324$ & \textbf{6} & $2^{2}$\\
	29 & \texttt{F} & \texttt{7} & \texttt{(9F131)} & \texttt{(5851F)} & \texttt{(44CC5)} & $-316$ & \textbf{6} & $2^{2}$\\
	30 & \texttt{1} & \texttt{9} & \texttt{(BB773)} & \texttt{(11383)} & \texttt{(CC8D8)} & $-308$ & \textbf{6} & $2^{2}$\\
	31 & \texttt{5} & \texttt{D} & \texttt{(33833)} & \texttt{(D5917)} & \texttt{(6E272)} & $-284$ & \textbf{6} & $2^{2}$\\
	32 & \texttt{1} & \texttt{9} & \texttt{(3FB5D)} & \texttt{(53835)} & \texttt{(22CDC)} & $-276$ & \textbf{6} & $2^{2}$\\
	33 & \texttt{3} & \texttt{B} & \texttt{(D13F1)} & \texttt{(9A95F)} & \texttt{(62126)} & $-264$ & \textbf{6} & $2^{2}$\\
	34 & \texttt{1} & \texttt{9} & \texttt{(B8BBB)} & \texttt{(175F9)} & \texttt{(AAEDE)} & $-368$ & \textbf{8} & $2^{2}$\\
	35 & \texttt{B} & \texttt{B} & \texttt{(97DB7)} & \texttt{(87575)} & \texttt{(E6CDC)} & $-352$ & \textbf{8} & $2^{2}$\\
	36 & \texttt{F} & \texttt{9} & \texttt{(77F7F)} & \texttt{(81DBF)} & \texttt{(E7EEE)} & $-292$ & \textbf{8} & $2^{2}$\\
	37 & \texttt{3} & \texttt{B} & \texttt{(D7383)} & \texttt{(1F5B9)} & \texttt{(474C4)} & $-276$ & \textbf{8} & $2^{2}$\\
	38 & \texttt{5} & \texttt{B} & \texttt{(7F4FF)} & \texttt{(5717B)} & \texttt{(8888D)} & $-384$ & 18 & $2^{2}$\\
	39 & \texttt{1} & \texttt{9} & \texttt{(BB71B)} & \texttt{(FF109)} & \texttt{(6F666)} & $-376$ & 18 & $2^{2}$\\
	40 & \texttt{F} & \texttt{7} & \texttt{(89F17)} & \texttt{(FF993)} & \texttt{(266A7)} & $-364$ & 18 & $2^{2}$\\
	41 & \texttt{5} & \texttt{D} & \texttt{(15D2D)} & \texttt{(17733)} & \texttt{(252E6)} & $-360$ & 18 & $2^{2}$\\\midrule
	\end{tabular}
	\end{adjustbox}
	\end{table}
	
	\begin{table}[h!]
	\caption{New singly-even binary self-dual $[80,40,14]$ codes from Theorem \ref{thm-1} over $\FF_4$.}\label{table-80-4}
	\centering
	\begin{adjustbox}{max width=\textwidth}
	\footnotesize
	\begin{tabular}{ccccccc}\midrule
	$\mathcal{C}_{80,i}$ & $\vct{a}$ & $\vct{b}$ & $\vct{c}$ & $\alpha$ & $\beta$ & $|\aut{\mathcal{C}_{80,i}}|$\\\midrule
	42 & \texttt{(2023113102)} & \texttt{(1313111210)} & \texttt{(1110131210)} & $-460$ & 0 & $2^{2}\cdot 5$\\
	43 & \texttt{(3121330000)} & \texttt{(2033021032)} & \texttt{(0320320302)} & $-290$ & \textbf{5} & $2\cdot 5$\\
	44 & \texttt{(0313022300)} & \texttt{(2111212100)} & \texttt{(1002013132)} & $-260$ & \textbf{5} & $2^{2}\cdot 5$\\
	45 & \texttt{(2231100000)} & \texttt{(1230010333)} & \texttt{(2001231310)} & $-250$ & \textbf{5} & $2^{2}\cdot 5$\\
	46 & \texttt{(2311211101)} & \texttt{(0100120321)} & \texttt{(0100010203)} & $-240$ & \textbf{5} & $2\cdot 5$\\
	47 & \texttt{(2312300112)} & \texttt{(0022222123)} & \texttt{(3333103120)} & $-220$ & \textbf{5} & $2\cdot 5$\\
	48 & \texttt{(0303110323)} & \texttt{(1231311133)} & \texttt{(0230132110)} & $-210$ & \textbf{5} & $2^{2}\cdot 5$\\
	49 & \texttt{(3213321031)} & \texttt{(0010131103)} & \texttt{(1203122220)} & $-200$ & \textbf{5} & $2\cdot 5$\\
	50 & \texttt{(0203231321)} & \texttt{(2031110022)} & \texttt{(0232110031)} & $-190$ & \textbf{5} & $2\cdot 5$\\\midrule
	\end{tabular}
	\end{adjustbox}
	\end{table}

\subsection{New Binary Self-Dual \BM{[92,46,16]} Codes}

The possible weight enumerators of a binary self-dual $[92,46,16]$ code are given in \cite{27} as
	\begin{align*}
	W_{92,1}&=1+(4692+4\alpha)x^{16}+(174800-8\alpha+256\beta)x^{18}\\&\quad+(2425488-52\alpha-2048\beta)x^{20}+\cdots,\\
	W_{92,2}&=1+(4692+4\alpha)x^{16}+(174800-8\alpha+256\beta)x^{18}\\&\quad+(2441872-52\alpha-2048\beta)x^{20}+\cdots,\\
	W_{92,3}&=1+(4692+4\alpha)x^{16}+(121296-8\alpha)x^{18}\\&\quad+(3213968-52\alpha)x^{20}+\cdots,
	\end{align*}
where $\alpha,\beta\in\ZZ$. The existence of codes with weight enumerator $W_{92,1}$ has previously been determined for
	\begin{enumerate}[label=]
	\item $\beta=-209$ and $\alpha\in\{3z:z=554,\lb569,\lb584,\lb614,\lb659,\lb689,\lb719,\lb869\}$;
	\item $\beta=-92$ and $\alpha\in\{23z:z=108\}$;
	\item $\beta=-69$ and $\alpha\in\{23z:z=88\}$;
	\item $\beta=-46$ and $\alpha\in\{23z:z=62,\lb64,\lb66,\lb67,\lb68,\lb70,\lb...,\lb78,\lb80,\lb81,\lb84,\lb86,\lb88\}$;
	\item $\beta=-23$ and $\alpha\in\{23z:z=46,\lb47,\lb49,\lb50,\lb52,\lb...,\lb56,\lb58,\lb...,\lb66,\lb68,\lb69\}$;
	\item $\beta=0$ and $\alpha\in\{23z:z=26,\lb33,\lb...,\lb56,\lb58,\lb60,\lb...,\lb63,\lb114\}$
	\end{enumerate}	
(see \cite{46,7,47,48,26}).

We obtain 12 new extremal binary self-dual codes of length 92 which have weight enumerator $W_{92,1}$ for
	\begin{enumerate}[label=]
	\item $\beta=-69$ and $\alpha\in\{23z:z=78\}$;
	\item $\beta=-46$ and $\alpha\in\{23z:z=61,\lb69,\lb79,\lb82,\lb85\}$;
	\item $\beta=-23$ and $\alpha\in\{23z:z=48,\lb57,\lb67,\lb76\}$;
	\item $\beta=0$ and $\alpha\in\{23z:z=31,\lb57\}$.
	\end{enumerate}

The new codes are constructed by applying Theorem \ref{thm-1} over $\FF_2$ (Table \ref{table-92-1}).

	\begin{table}[h!]
	\caption{New binary self-dual $[92,46,16]$ codes from Theorem \ref{thm-1} over $\FF_2$.}\label{table-92-1}
	\centering
	\begin{adjustbox}{max width=\textwidth}
	\footnotesize
	\begin{tabular}{cccccccc}\midrule
	$\mathcal{C}_{92,i}$ & $\vct{a}$ & $\vct{b}$ & $\vct{c}$ & $W_{92,j}$ & $\alpha$ & $\beta$ & $|\aut{\mathcal{C}_{92,i}}|$\\\midrule
	1 & \texttt{(00001001000010001001111)} & \texttt{(10101111101000110001110)} & \texttt{(00011010000001011011100)} & 1 & $1794$ & $-69$ & $2\cdot 23$\\
	2 & \texttt{(00001011001101011000110)} & \texttt{(00101011100111001111111)} & \texttt{(00000000000000000001000)} & 1 & $1403$ & $-46$ & $2\cdot 23$\\
	3 & \texttt{(10111011011110001100111)} & \texttt{(01000011000001010111101)} & \texttt{(11111001011011100000101)} & 1 & $1587$ & $-46$ & $2\cdot 23$\\
	4 & \texttt{(10111000101000001011000)} & \texttt{(01100110111111101001111)} & \texttt{(10101000100101111011101)} & 1 & $1817$ & $-46$ & $2\cdot 23$\\
	5 & \texttt{(11111010100111001110011)} & \texttt{(11110100110110011010011)} & \texttt{(00000000000000000100000)} & 1 & $1886$ & $-46$ & $2\cdot 23$\\
	6 & \texttt{(01011100110111010011100)} & \texttt{(01001000010010010101100)} & \texttt{(10001010011110100011111)} & 1 & $1955$ & $-46$ & $23$\\
	7 & \texttt{(10110110101001111101110)} & \texttt{(01100100101000011001110)} & \texttt{(01010110111011011000110)} & 1 & $1104$ & $-23$ & $23$\\
	8 & \texttt{(10000001101101000000011)} & \texttt{(10011010001000111111011)} & \texttt{(00000000000001000000000)} & 1 & $1311$ & $-23$ & $2\cdot 23$\\
	9 & \texttt{(00101001110111010101110)} & \texttt{(10000011110100111101010)} & \texttt{(00000000000000000010000)} & 1 & $1541$ & $-23$ & $2\cdot 23$\\
	10 & \texttt{(11011001100001010001000)} & \texttt{(00101000101111101010110)} & \texttt{(00100000000000000000000)} & 1 & $1748$ & $-23$ & $2\cdot 23$\\
	11 & \texttt{(00111010000001101000100)} & \texttt{(01100001101101000100001)} & \texttt{(11000001011111100101101)} & 1 & $713$ & $0$ & $2\cdot 23$\\
	12 & \texttt{(01101110111110110010111)} & \texttt{(00111011110010101111000)} & \texttt{(00000000000000010000000)} & 1 & $1311$ & $0$ & $2\cdot 23$\\\midrule
	\end{tabular}
	\end{adjustbox}
	\end{table}

\section{Conclusion}

In this work, we presented a technique for constructing self-dual codes which was derived as a modification of the four circulant construction utilising $\lambda$-circulant matrices. We proved the necessary conditions required for this technique to produce self-dual codes using a specialised mapping $\Theta$ related to $\lambda$-circulant matrices. We proved the ability of this technique by using it to construct many best known, optimal and extremal binary self-dual codes which were previously not known to exist. In particular, we have been able to construct new best known codes of length 56, a new optimal code of length 58, a new extremal code of length 64, many new best known codes of length 80, and 12 new extremal codes of length 92.

Because the search field of the given technique is considerably greater in size than that of the four circulant construction, the codes were obtained by random searches alone. As such there are potentially more new codes to be found by applying this technique (especially for lengths 80 and 92). On the other hand, by using $\Theta$ to reduce the number of computations required in our algorithms, we were able to increase the rate at which we could find self-dual codes. Due to computational limitations, we did not investigate constructing codes of lengths greater than 92 and so this could be a suggestion for future work. We could also investigate computationally feasible ways of classifying the codes produced by the technique for specific rings and lengths. Further still, we could investigate defining mappings similar to $\Theta$ whose implementation could potentially increase our computational efficiency even more. Another suggestion would be to consider different families of orthogonal matrices for $C$.
\bibliographystyle{plainnat}
\bibliography{paper1}
\end{document}